\documentclass[11pt]{amsart}  

\usepackage{amsmath}
\usepackage{amsthm}
\usepackage{amsfonts}
\usepackage{amssymb}
\usepackage{enumitem}
\usepackage{eucal}
\usepackage[all]{xy}
\usepackage{amsxtra}
\usepackage{appendix} 
\usepackage{tensor} 
\usepackage{url}
\usepackage{upgreek}
\usepackage{etoolbox}
\apptocmd{\sloppy}{\hbadness 10000\relax}{}{}
\usepackage{comment}
\usepackage{mathtools}
\usepackage{hyperref}

\RequirePackage{color}
\definecolor{bwgreen}{rgb}{0.183,1,0.5}
\definecolor{bwmagenta}{rgb}{0.7,0.0,0.1}
\definecolor{bwblue}{rgb}{0.317,0.161,1}

\hypersetup{
pdfauthor = {\textcopyright\ Bryden Cais},
pdfcreator = {\LaTeX\ with package \flqq hyperref\frqq},
colorlinks = {true},
linkcolor={bwmagenta},	 %color red Color for normal internal links.
anchorcolor={},	 %color	black Color for anchor text.
citecolor={bwblue},	% color	green Color for bibliographical citations in text.
filecolor={},	% color cyan Color for URLs which open local files.
menucolor={},	% color	red	Color for Acrobat menu items.
runcolor={},	% color	filecolor Color for run links (launch annotations).
urlcolor={bwgreen}, %	 color	magenta	Color for linked URLs.
}

\CompileMatrices
\entrymodifiers={+!!<0pt,\fontdimen22\textfont2>}
\DeclareFontFamily{OT1}{rsfs}{}
\DeclareFontShape{OT1}{rsfs}{n}{it}{<-> rsfs10}{}
\DeclareMathAlphabet{\mathscr}{OT1}{rsfs}{n}{it}

\DeclareFontFamily{OT1}{pzc}{}
\DeclareFontShape{OT1}{pzc}{n}{it}{<->s*[2.2]pzc}{}
\DeclareMathAlphabet{\mathpzc}{OT1}{pzc}{b}{sl}

\makeatletter

\newcommand{\Rmnum}[1]{\expandafter\@slowromancap\romannumeral #1@}
\makeatother

\DeclareMathOperator{\Gal}{Gal}

\DeclareMathOperator{\sep}{sep}

\DeclareMathOperator{\gal}{gal}
\newcommand*{\R}{\ensuremath{\mathbf{R}}}   
              
\newcommand*{\Z}{\ensuremath{\mathbf{Z}}}               
\newcommand*{\Q}{\ensuremath{\mathbf{Q}}}

\newcommand*{\Kbar}{\overline{K}}    
          
\newcommand*{\C}{\mathbf{C}}
\newcommand*{\E}{\mathscr{E}}

\let\danishO\O
\let\O\tempO
 
\newcommand*{\e}{\ensuremath{\mathbf{E}}}
\renewcommand*{\a}{\ensuremath{\mathbf{A}}}

\newcommand*{\wt}[1]{\widetilde{#1}}

\numberwithin{equation}{section}
\theoremstyle{plain}
  \newtheorem{theorem}[equation]{Theorem}
    
  \newtheorem{proposition}[equation]{Proposition}
  \newtheorem{lemma}[equation]{Lemma}
  \newtheorem{corollary}[equation]{Corollary}

\theoremstyle{definition}
  \newtheorem{definition}[equation]{Definition}

\theoremstyle{remark}
  \newtheorem{example}[equation]{Example}
  \newtheorem{remark}[equation]{Remark}

\begin{document}
\title{A characterization of strictly APF extensions}

\author{Bryden Cais}
\address{University of Arizona, Tucson}
\curraddr{Department of Mathematics, 617 N. Santa Rita Ave., Tucson AZ. 85721}
\email{cais@math.arizona.edu}

\author{Christopher Davis}
\address{University of Copenhagen, Denmark}
\curraddr{Department of Mathematical Sciences,  Universitetsparken~5, 
DK-2100 K{\o}benhavn \danishO}
\email{davis@math.ku.dk}

\author{Jonathan Lubin}
\address{Brown University, Providence}
\curraddr{Department of Mathematics, Box 1917, Providence RI. 02912}
\email{lubinj@math.brown.edu}

%\thanks{
%	During the writing of this paper, the author was partially supported by an NSA Young Investigator grant
%	(H98230-12-1-0238).
%	}

%\dedicatory{}

\subjclass[2010]{Primary:  Secondary: }
\date{\today}

\begin{abstract}
Let $K$ denote a finite extension of $\Q_p$.  
We give necessary and sufficient conditions for an infinite totally wildly ramified extension $L/K$ to be strictly APF in the sense of Fontaine-Wintenberger.  Our conditions are phrased in terms of the existence of a certain tower of intermediate subfields.  
These conditions are well-suited to producing examples of strictly APF extensions, and in particular, our main theorem  proves that the \emph{$\varphi$-iterate} extensions previously considered by the first two authors are strictly APF.
%These conditions are well-suited to producing examples of strictly APF extensions.  In particular we prove that the ``strictly APF'' hypothesis appearing in the definition of $\varphi$-iterate extensions (as previously considered by the first two authors) is unnecessary; the procedure for producing such extensions always yields a strictly APF extension.  \chris{Thoughts on this last sentence?  I could imagine deleting it.  I couldn't find a more concise way to phrase it.}
\end{abstract}

\maketitle

\section{Introduction}

Let $p$ be a prime and $K$ a finite extension of $\Q_p$ with residue field $k$
and valuation $v_K$ normalized so that $v_K(K^{\times})=\Z$.
Fix an algebraic closure $\Kbar$ of $K$, and for any subfield $E$ of $\Kbar$
containing $K$ write $G_E:=\Gal(\Kbar/E)$.
Recall \cite{Wintenberger} that an infinite, totally wildly ramified extension $L/K$ is said
to be {\em arithmetically profinite} (APF) if the upper numbering ramification groups
$G_K^uG_L$ are {\em open} in $G_K$ for all $u\ge 0$.  
The field of norms machinery of Fontaine--Wintenberger \cite{Wintenberger}
functorially associates to any such APF extension $L/K$ a complete, discretely valued
field $X_K(L)$ of equicharacteristic $p$ and residue field $k$
with the amazing property that the \'etale sites of $L$ and $X_K(L)$
are equivalent; in particular, one has a canonical isomorphism of topological
groups $G_L\simeq \Gal(X_K(L)^{\sep}/X_K(L))$ that is compatible
with the upper numbering ramification filtrations.  In certain special 
cases, this isomorphism plays a foundational role in Fontaine's theory
of $(\varphi,\Gamma)$-modules \cite{Fontaine90} and in the integral $p$-adic Hodge theory 
of Faltings \cite{Faltings}, Breuil \cite{BreuilNormes,BreuilIntegral}, and Kisin \cite{KisinCrystal},
and in general provides a key ingredient of Scholze's recent theory
of perfectoid spaces and tilting \cite{Scholze}.

A famous theorem of Sen \cite{Sen72} guarantees that any infinite, totally wildly ramified
{\em Galois} extension $L/K$ with $\Gal(L/K)$ a $p$-adic Lie group 
is strictly\footnote{The meaning of the {\em strictness} condition, whose definition (Definition~\ref{strAPF})
is somewhat technical, is that
the norm field $X_K(L)$ of $L/K$ admits a canonical embedding into the fraction field $\wt{\e}$
of Fontaine's ring $\wt{\e}^+:=\varprojlim_{x\mapsto x^p} \O_{\C_K}/(p)$;
see also Remark \ref{slope remark} for a geometric interpretation 
of strictness.} 
APF; however, there are many other interesting and important cases in which 
one is given an infinite and totally wildly ramified extension $L/K$,
and one would like to decide whether or not $L/K$ is strictly APF.
Such examples occur naturally in the theory of $p$-adic analytic dynamics as follows:
%Fix an algebraic closure $\Kbar$ of $K$, and 
Choosing a uniformizer $\pi_1$ of $K$,
let $\varphi\in \O_K[\![x]\!]$ be a power series which reduces modulo $\pi_1$ to some power of
the Frobenius endomorphism of $k[\![x]\!]$ and which fixes zero, and let $\{\pi_n\}_{n\ge 1}$
be a compatible system (i.e., $\varphi(\pi_{n})=\pi_{n-1}$) of choices of roots
of $\varphi^{(n)}-\pi_1$.  The arithmetic of the rising union
$L:=\cup_{n\ge 1} K(\pi_n)$ is of serious interest
(e.g., \cite{LubinDynamics}).  For example, if $G$ is a Lubin--Tate formal group
over (the valuation ring of) a subfield $F$ of $K$ and $\varphi$
is the power series giving multiplication by a uniformizer of $F$,
then one may choose $\{\pi_n\}_{n\ge 1}$ so that $L/K$ is the Lubin--Tate extension
generated by the $p$-power torsion points of $G$ in $\Kbar$.
While it is true that $L/K$ is strictly APF if its Galois closure $L^{\gal}/K$
is \cite[Proposition~1.2.3(iii)]{Wintenberger}, it is often very difficult or impossible in practice
to describe $\Gal(L^{\gal}/K)$, and so Sen's theorem is of limited use in these cases.

In this note, we establish the following elementary and explicit characterization 
of strictly APF extensions:

\begin{theorem} \label{main theorem}
Let $L/K$ be an infinite, totally wildly ramified extension.
Then $L/K$ is strictly APF if and only if there exists a tower of finite
extensions $\{E_n\}_{n\ge 2}$ of $E_1 := K$ inside $L$
with $L=\cup E_n$ and a norm compatible sequence $\{\pi_n\}_{n\ge 1}$ with $\pi_{n}$ a uniformizer of $E_n$
such that: 
\begin{enumerate}
	\item The degrees $q_n := [E_{n+1}:E_n]$ are bounded above. \label{bdd}
	\item If $f_n(x)  = x^{q_n} + a_{n,q_n-1}x^{q_n-1} + \cdots + a_{n,1} x + (-1)^p \pi_{n}\in E_{n}[x]$ 
is the minimal polynomial of $\pi_{n+1}$ over $E_{n}$,
	then the non-constant and non-leading coefficients $a_{n,i}$ of $f_n$ satisfy $v_K(a_{n,i}) > \epsilon$ for some $\epsilon > 0$, independent of $n$ and $i$. \label{coef}
\end{enumerate}
Moreover, if $L/K$ is strictly APF, one may take $\{E_n\}$ to be the tower of
elementary subextensions {\rm (}see Definition~\ref{elem subfield def}{\rm)} and $\{\pi_n\}$ to be any norm-compatible sequence of uniformizers.
%In particular, functorially attached to $L/K$ is its {\em norm field} $X_K(L)$, which 
%is a complete, discretely valued field of equicharacteristic $p$, and there is a canonical 
%isomorphism of topological groups 
%$$\Gal(\Kbar/L) \simeq \Gal(X_K(L)^{\sep}/X_K(L))$$
%that identifies the upper-numbering ramification filtrations.
\end{theorem}

As a consequence of our work, we are able
to produce many concrete examples of (typically non-Galois) strictly APF extensions
as follows:
	let $\pi_1$ be a uniformizer of $E_1:=K$; for $n\ge 1$ and given $E_n$ and $\pi_n \in E_n$ a uniformizer,
	choose a monic polynomial $\varphi_n(x) \in \O_{E_n}[x]$ satisfying $\varphi_n(0) = 0$ and $\varphi_n(x) \equiv x^{q_n} \bmod \pi_{n} \O_{E_n}$ for $q_n > 1$ a power of $p$, and let $\pi_n$ be a choice of root of 
	$f_n(x):=\varphi_n(x)-\pi_{n-1} = 0$.   If the degrees $q_n$ are bounded above and the non-leading and non-constant coefficients
	of the $f_n$ have $v_K$-valuation bounded below, then
	it follows at once from Theorem \ref{main theorem} that
 $L := \cup_{n} E_n$ is a strictly APF extension of $K$.
 In particular, the ``$\varphi$-iterate" extensions 
described above are always strictly APF.

%The meaning of the {\em strictness} condition, whose definition is somewhat technical (see below) 
%is that the norm field $X_K(L)$ of $L/K$ admits a canonical embedding into the fraction field $\wt{\e}$
%of Fontaine's ring $\wt{\e}^+:=\varprojlim_{x\mapsto x^p} \O_{\C_K}/(p)$.  As such,

In \S\ref{examples section}, we provide several examples of infinite, totally ramified
extensions $L/K$ which are not APF, or which are APF but not strictly
APF, to illustrate the subtlety of these conditions.  

As any strictly APF extension $L/K$ has norm field $X_K(L)$
that is canonically identified with a subfield of Fontaine's field $\wt{\e}$, 
one can try to find a canonical and functorial
lift of $X_K(L)$ to a subring of $\wt{\a}:=W(\wt{\e})$.
Such lifts play a crucial role in Fontaine's classification \cite{Fontaine90}
of $p$-adic representations of $G_L$ by \'etale $\varphi$-modules.
%and, for $f_n(x) =(1+x)^p-1$, $K=\Q_p(\mu_p)$ and $L=\Q_p(\mu_{p^{\infty}})$,
%are foundational in the theory of $(\varphi,\Gamma)$-modules.  
The question of functorially lifting $X_K(L)$
inside $\wt{\a}$ is studied in \cite{CaisDavis} and \cite{Berger},
and the main theorem of the present paper provides essential input for \cite{CaisDavis}.

\begin{remark} \label{char p remark}
Much of Theorem~\ref{main theorem} continues to hold if we allow $K$ to be an \emph{equicharacteristic} local field.  In particular, for $\{E_n\}$ satisfying Conditions~(\ref{bdd}) and (\ref{coef}), the field $L := \cup E_n$ is a strictly APF extension of $K$.  Conversely, for $L/K$ infinite totally wildly ramified and strictly APF and for $\{E_n\}$ the tower of elementary subextensions and $\{\pi_n\}$ a norm compatible sequence of uniformizers, Condition~(\ref{coef}) continues to hold.   (The proofs given below in the mixed characteristic case work in the equal characteristic case as well.)  However, Example~\ref{char p example} below shows we cannot expect Condition~(\ref{bdd}) to hold in general.  
\end{remark}

\begin{remark}
The proofs below produce an explicit lower bound for the constant $c(L/K)$ appearing in the definition of strictly APF (Definition~\ref{strAPF}).  The lower bound is given in terms of $\max q_n$ and $\epsilon$ as in Theorem~\ref{main theorem}.
\end{remark}

\section{Transition functions and ramification} \label{transition section}

Following \cite[\S2]{Lubin}, we briefly review the definition and properties of the Herbrand transition 
functions, and recall the definitions of APF and strictly APF as in \cite[1.2]{Wintenberger}.

Let $L/K$ be a finite, totally ramified extension contained in $\Kbar$,
and $\pi_L$ a uniformizer of $L$.
Write $v_L$ for the valuation of $\Kbar$, normalized so that $v_L(\pi_L)=1$.
Denote by $G$ the Galois set of all $K$-embeddings of $L$ into $\Kbar$,
and for real $t\ge 0$ set 
$$
G_t:=\{\sigma\in G\ :\ v_L(\sigma(\pi_L)-\pi_L)\ge t\}.
$$
We define the transition function
%$\gamma_t$ the number of
%$K$-embeddings $\sigma:L\hookrightarrow \o{K}$ satisfying $v_L(\sigma(\pi_L)-\pi_L)\ge t$,
%and define
\begin{equation*}
	\phi_{L/K}(u):=\frac{1}{[L:K]}\int_{0}^u |G_t| \, dt;
\end{equation*}
it is a continuous, piecewise linear and increasing bijection on $[0,\infty)$,
so we may define $\psi_{L/K}:=\phi_{L/K}^{-1}$, which
is again continuous, piecewise linear and increasing.  For $L'/L$ any finite, totally ramified
extension contained in $\Kbar$, one has the transitivity relations
\begin{equation}
	\phi_{L'/K} = \phi_{L/K}\circ \phi_{L'/L}\qquad\text{and}\qquad
	\psi_{L'/K} = \psi_{L'/L} \circ \psi_{L/K}.\label{transitivity}
\end{equation}
In practice, we may compute $\phi_{L'/L}$ as follows:

\begin{proposition}[{\cite[Lemma~1]{Lubin}}] \label{copoly and phi}
	Let $L'/L$ be a finite, totally ramified extension of subfields of $\Kbar$
	containing $K$.  Choose a uniformizer $\pi_{L'}$ of $L'$ and let $f(x)\in L[x]$ be the minimal polynomial
	of $\pi_{L'}$ over $L$.  Set $g(x):=f(x+\pi_{L'})\in L'[x]$, and let
	$\Psi_g$ be the function whose graph is the boundary of the Newton copolygon of $g(x)=\sum_{n\ge 1} b_nx^n$
	formed by the intersection of the half-planes $\{y\le ix + v_K(b_i)\}_{i\ge 1}$.  Then 
	\begin{equation}
		\phi_{L'/L}(x) = e_{L/K}\Psi_g(x/e_{L'/K}).
	\end{equation}
\end{proposition}

If $L/K$ is finite Galois, then the $G_t$ are the usual
lower-numbering ramification subgroups of $G$, and we define
the ramification subgroups in the {\em upper-numbering}
to be $G^t:=G_{\psi_{L/K}(t)}$.  Unlike the lower-numbering
groups, the $G^t$ are well-behaved with respect to quotients:
if $K'$ is a finite Galois extension of $K$ contained in $L$
then for $H:=\Gal(L/K') \mathrel{\unlhd} G$ one has $(G/H)^t = G^tH/H$
for all real $t\ge 0$.  It follows that by taking projective limits,
we may define the upper numbering filtration $\{G^t\}_{t\ge 0}$
for {\em any} Galois extension $L/K$, finite or infinite, contained
in $\Kbar$; this is a separated and exhaustive 
decreasing filtration of $G$ by closed normal subgroups.

\begin{remark}\label{Lshift}
	Because of our desire to have the simple description of $\phi_{L'/L}$
	given in Proposition \ref{copoly and phi}, our transition functions differ from the
	ones considered by Serre \cite{Serre79} and Wintenberger \cite{Wintenberger}
	by a shift.  Indeed, following \cite[\S2]{Lubin}, if $\prescript{}{S}{\phi}_{L'/L}$
	and $\prescript{}{S}{\psi}_{L'/L}$ denote the transition functions
	defined by Serre \cite[\Rmnum{4} \S3]{Serre79}, then one has the relations
	\begin{equation*}
		\phi_{L'/L}(x)=1+\prescript{}{S}{\phi}_{L'/L}(x-1)\qquad\text{and}\qquad
		\psi_{L'/L}(x)=1+\prescript{}{S}{\psi}_{L'/L}(x-1).
	\end{equation*}
	Correspondingly, the relation between our ramification groups
	$G_t$ and $G^t$
	and those defined by Serre $\prescript{}{S}{G}_t$,
	$\prescript{}{S}{G}^t$ is a through shift of one:
%	\begin{equation*}
$G_t = \prescript{}{S}{G}_{t-1}$ and
		$G^t = \prescript{}{S}{G}^{t-1}$.  
%	\end{equation*}
\end{remark}

For any extension $E$ of $K$ contained in $\Kbar$, we %write $G_E:=\Gal(\Kbar/E)$,
define
\begin{equation}
i(E/K):=\sup_{t\ge 0}\{t\ :\ G_K^tG_E=G_K\}.\label{idef}
\end{equation}

\begin{definition}\label{APFDEF}
	Let $L/K$ be an arbitrary $($possibly infinite$)$ totally ramified
	extension of $K$ contained in $\Kbar$.  
	We say that $L/K$ is {\em arithmetically profinite} $($APF$)$
	if $G_K^u G_L$ is open in $G_K$ for all $u\ge 0$.  If $L/K$ is APF,
	we define
	\begin{equation}
		\psi_{L/K}(u):= \int_0^u [G_K: G_K^vG_L] \, dv,\label{int}
	\end{equation}
	which is a continuous and piecewise linear increasing bijection on $[0,\infty)$,
	and we write $\phi_{L/K}:=\psi_{L/K}^{-1}$. 
\end{definition}

Observe that any finite totally ramified extension $L/K$ is APF, 
and the functions
$\phi_{L/K}$ and $\psi_{L/K}$ of Definition \ref{APFDEF} coincide
with the previously defined transition functions of the same name.
It follows from the definition that if $L/K$ is an infinite APF extension, 
then the set of ramification breaks $\{b\in \R_{\ge 0}\ :\ G_K^{b+\varepsilon}G_L\neq G_K^bG_L\ \forall\,\varepsilon>0\}$
is {\em discrete} and unbounded, so we may enumerate these real numbers as $b_1 < b_2 < \ldots$.
\begin{definition} \label{elem subfield def}
	The {\em $n$-th elementary subextension of $L/K$} is the subfield $K_n$ of $\Kbar$ fixed by $G_K^{b_n}G_L$.
\end{definition}
We note that each $K_n$ is a finite extension of $K$ contained in $L$, that $L$
is the rising union of the $K_n$, and that $K_{n+1}/K_n$
is {\em elementary of level $i_n$} for $i_n:=i(K_{n+1}/K_n)=i(L/K_n)$
in the sense that there is a unique break at $u=i_n$ in the filtration
$\{G_{K_n}^uG_{K_{n+1}}\}_{u\ge 0}$ of $G_{K_n}$. Equivalently, the transition function $\phi_{K_{n+1}/K_n}$
is the boundary function of the intersection of the two half-planes $y\le x$ and $y\le [K_{n+1}:K_n]^{-1}(x-i_n)+i_n$,
and has a single vertex at $(i_n,i_n)$.  As in \cite[1.4.1]{Wintenberger}, it follows that $\{i_n\}_{n\ge 1}$
is an increasing and unbounded sequence, and that one has
\begin{equation}
	b_n = i_1 + \frac{i_2-i_1}{[K_2:K_1]} + \frac{i_3-i_2}{[K_3:K_1]} + \cdots + 
	\frac{i_n-i_{n-1}}{[K_{n}:K_1]}.  \label{bformula}
\end{equation}
with $\{b_n\}_{n\ge 1}$ increasing and unbounded.  It follows easily from definitions
that the vertices of the function 
$\phi_{L/K}$ of Definition \ref{APFDEF} are $\{(i_n,b_n)\}_{n\ge 1}$, and the
slope of the segment immediately to the right of $(i_n,b_n)$ is $[K_{n+1}:K]^{-1}$.

We will make use of the following characterization:

\begin{proposition}\label{APF-char}
	Let $\{E_n\}_{n\ge 2}$ be a tower of finite extensions of $E_1 := K$ and let $L=\cup_{n\ge 1} E_n$
	be their rising union.  Set $\Phi_n:=\phi_{E_n/K}$ and define
	$\alpha_n:=\sup \{ x\ :\ \Phi_{n+1}(x)=\Phi_{n}(x)\}$.
	Then $L/K$ is APF if and only if the following two conditions hold:
	\begin{enumerate}
		\item We have $\lim_{n \rightarrow \infty} \alpha_n = \infty$. In particular, the pointwise limit $\Phi(x) := \lim_{n \rightarrow \infty} \Phi_n(x)$ exists, and moreover,
		for fixed $x_1$, we have $\Phi(x) = \Phi_n(x) $ for all $x \leq x_1$ and all $n$ sufficiently large. \label{converge}
		\item The function $\Phi(x)$ of $(\ref{converge})$ is piecewise linear and continuous,
	 with vertices $\{(i_n,b_n)\}_{n\ge 1}$ where $\{i_n\}$
		and $\{b_n\}$ increasing and unbounded sequences. \label{vertices increasing condition}
	\end{enumerate}
	If $L/K$ is APF, then $\Phi(x)=\phi_{L/K}$ for
	$\phi_{L/K}$ as in Definition $(\ref{APFDEF})$.
\end{proposition}

\begin{proof}
Assume first that the two numbered conditions hold.  From the assumption that the $\{b_n\}$ sequence is unbounded, we know the inverse function $\Phi^{-1}(x)$ is defined for all $x \geq 0$ and is the pointwise limit of $\Phi_n^{-1}(x)$ (for any $x$, we have $\Phi^{-1}(x) = \Phi_n^{-1}(x)$ for all $n$ suitably large).  
By definition, $\Phi_n^{-1}(x) = \phi_{E_n/K}^{-1}(x) = \psi_{E_n/K}(x)$.  Thus, the convergence condition (and the definition of $\psi$) implies that for any $u$ we have 
$[G_K : G_K^u G_{E_n} ] = [G_K : G_K^u G_{E_{n+1}} ]$ for all $n$ suitably large.  Writing momentarily $K'$ for the fixed field of $G_K^u$ acting on $\Kbar$, it follows that $K' \cap E_n = K' \cap E_{n+1}$ for all $n$ sufficiently large.  Hence this intersection is also equal to $K' \cap L$ and so, for fixed $u$, 
we find $[G_K : G_K^u G_{E_n} ] = [G_K :  G_K^uG_L]$ for $n$ suitably large.  In particular,
$G_LG_K^u$ is of finite index---and hence open---in $G_K$ for every $u$,
and $L/K$ is APF.
%This shows that $\Phi^{-1}(x) = \psi_{L/K}(x)$ and $\Phi(x) = \phi_{L/K}(x)$.   The fact that $L/K$ is APF now follows from (\ref{vertices increasing condition}).

Now assume $L/K$ is APF, and  let $\{K_n\}$ be the associated tower of elementary extensions as in Definition~\ref{elem subfield def}.  By \cite[1.4.1]{Wintenberger}, we have $\lim_{n \rightarrow \infty} i(L/K_n) = \infty$.  This implies that for any fixed $u$, 
there exists $n_0:=n_0(u)$ with
$[G_{K_n} : G_{K_n}^u G_{L} ] = 1$ and hence $\psi_{L/K_n}(u) = u$ for all $n \ge n_0$.
As $L=\cup E_m$, for any $u$ there exists $m_0=m_0(u)$ with $E_m\supseteq K_{n_0(u)}$
whenever $m \ge m_0(u)$.
We then have $\alpha_{m+1} \geq u$ for all 
$m\ge m_0(u)$; as $u$ was arbitrary, this implies (\ref{converge}).
It follows that $\Phi:=\lim_{n\rightarrow\infty} \Phi_n$ is well-defined, piecewise linear and continuous, and is the unique
such function with $\Phi'(u)=[G_K: G_K^uG_L]^{-1}$ whenever $u$ is not the $x$-coordinate of 
a vertex.  In particular, $\Phi=\phi_{L/K}$ for $\phi_{L/K}$
as in Definition $\ref{APFDEF}$; since $L/K$ is APF we conclude that (\ref{vertices increasing condition})
holds.
%We showed in the first paragraph of the proof that, provided (\ref{converge}) holds, $\Phi(x) = \phi_{L/K}(x)$.  If $L/K$ is APF, the properties of (\ref{vertices increasing condition}) clearly hold for $\phi_{L/K}(x)$.
\end{proof}

\begin{corollary}[{\cite[1.4.2]{Wintenberger}}]\label{WintCor}
Set $E_1 := K$ and 
	for $n \geq 1$, assume that $E_{n+1}/E_n$ is elementary of level $i_n$
	with $\{i_n\}$ strictly increasing and unbounded, and let $\{b_n\}$
	be given by $(\ref{bformula})$.  Then $L:=\cup_n E_n$
	is an APF extension of $K$ if and only if $\{b_n\}$
	is unbounded. Moreover, if $L/K$ is APF, then $E_n$ is the $n$-th elementary subextension of
	 $L/K$ as in Definition~\ref{elem subfield def}.
\end{corollary}

\begin{definition}[{\cite[1.4.1]{Wintenberger}}]\label{strAPF}
	Let $L/K$ be an infinite APF extension with associated elementary tower $\{K_n\}$,
	and recall the function $i(\cdot)$ of $(\ref{idef}).$ We define 
	\begin{equation*}
		c(L/K):=\inf_{u \ge i(L/K)} \frac{\psi_{L/K}(u)}{[G_K:G_K^uG_L]}
		=\inf \frac{i_n}{[K_{n+1}:K]}
	\end{equation*}
	for $i_n:=i(K_{n+1}/K_n)=i(L/K_n)$. 
	%\chris{Note for Bryden: I guess because the degrees are unbounded, the ``shift by 1'' discrepancy between our definition and Wintenberger's definition should not affect $c$.} 
	We say that $L/K$ is {\em strictly APF}
	if $c(L/K)>0$. 
\end{definition}

\begin{remark} \label{slope remark}
	If $L/K$ is an infinite APF extension, it follows immediately from 
	Definition \ref{strAPF} and the discussion preceding Proposition \ref{APF-char}
	that the constant $c(L/K)$ is equal to $\inf v_n m_{n}$
	where $v_n$ is the $x$-coordinate of the $n$-th vertex of $\phi_{L/K}$
	and $m_n$ is the slope of the segment of $\phi_{L/K}$ immediately to the right
	of $v_n$.  Thus, $L/K$ is strictly APF if and only if the sequence $\{v_n m_n\}$
	is bounded below by a constant $c > 0$.
	More geometrically, the strictness condition is equivalent to 
	$ [G_K: G_K^uG_L]^{-1} \ge c/u $ for $u\ge i(L/K)$, which, upon integrating, is equivalent to
	the bound
	\begin{equation*}
		\phi_{L/K}(x) \ge c \log(x) + d\qquad\text{for}\qquad d:=i(L/K) - c\log(i(L/K))
	\end{equation*}
	for all $x\ge i(L/K)$.
\end{remark}

\begin{lemma} \label{comparison V and V_n}
Let $\{E_n\}_{n\ge 2}$ be a tower of finite extensions of $E_1 := K$
and $L:=\cup_n E_n$.  Suppose that $L/K$ is APF, and let 
$\Phi$ and $\Phi_n$ be the transition functions of Proposition \ref{APF-char}.
Let $V_n$ be the set of $x$-coordinates of vertices of $\Phi_n$,
and for $v\in V_n$ let $m_v$ be the slope of the segment of $\Phi_n$
immediately to the right of $v$.  Then
%Let $V$ and $V_n$ (respectively) denote the sets of $x$-coordinates of vertices of $\Phi(x)$ and $\Phi_n(x)$ defined in Proposition~\ref{APF-char}.  For $v \in V$ or $V_n$, let $m_v$ denote the slope of the segment to the immediate right of $v$.  
%We have  
\[
c(L/K)\geq \liminf_{n \rightarrow \infty} \left(\min_{v \in V_n} v m_v \right).
\]
\end{lemma}

\begin{proof}
Writing $V$ for the set of $x$-coordinates of vertices of $\Phi$, we have
$c(L/K) = \inf_{v \in V} v m_v$ by Remark \ref{slope remark}.  This means that for any $\epsilon > 0$, we can find $v \in V$ such that $vm_v < c(L/K) + \epsilon$.  It follows from Proposition~\ref{APF-char}(\ref{converge})
that any vertex $v$ of $\Phi$ is a vertex of $\Phi_n$ for all $n$
sufficiently large, and the slopes of the segments on $\Phi$
and $\Phi_n$ to the immediate right of $v$ agree.
Thus 
$\min_{v \in V_n} v m_v < c(L/K) + \epsilon$ for all $n$ sufficiently large, which completes the proof.
\end{proof}

\section{Proof of Theorem~\ref{main theorem}}

From now until the end of Proposition~\ref{sufficient direction}, fix an infinite totally wildly ramified extension $L/K$ with a tower of subextensions $\{E_n\}$ satisfying Conditions~(\ref{bdd}) and (\ref{coef}) from Theorem~\ref{main theorem}.   We will show that such an extension $L/K$ is strictly APF, thus proving one direction of Theorem~\ref{main theorem}.  

\begin{lemma} \label{b valuations}
Let $f_n(x)$ and $\pi_n$ be as in Theorem~$\ref{main theorem}(\ref{coef})$.  Write 
\[
f_n(x) = x^{q_n} + a_{n,q_n-1}x^{q_n-1} + \cdots + a_{n,1} x + (-1)^p \pi_{n}, 
\]
so
\begin{equation} \label{coefficient notation}
	g_n(x):=f_n(x+\pi_{n+1}) = \sum_{i=1}^{q_n} b_{n,i} x^i,\quad\text{for}\quad
	b_{n,i}:=\sum_{j\ge i} a_{n,j} {{j}\choose{i}} \pi_{n+1}^{j-i}.
\end{equation}
Let $1 > \epsilon > 0$ be such that $v_K(a_{n,i}) > \epsilon$ for all $0 < i < q_n$.  If $0 < i < q_n$, then $v_K(b_{n,i}) > \epsilon.$ 
\end{lemma}

\begin{proof}
If $j \neq q_n$, then $v_K(a_{n,j}) > \epsilon$ by hypothesis and so $v_K\left(a_{n,j} {{j}\choose{i}} \pi_{n+1}^{j-i}\right) > \epsilon$.  If $j = q_n$ and $0 < i < q_n$, then 
$v_K{j \choose i} \geq v_K(p) \geq 1.$
\end{proof}

\begin{proposition} \label{it is APF}
The extension $L/K$ is APF.
%Assume $L/K$ is an infinite totally wildly ramified extension and $\{E_n\}$ is a tower of subextensions with $\cup E_n = L$ and satisfying Conditions~(\ref{bdd}) and~(\ref{coef}) of Theorem~\ref{main theorem}.  Then $L/K$ is APF.
\end{proposition}

\begin{proof}
We prove this by verifying Conditions~(\ref{converge}) and~(\ref{vertices increasing condition}) of Proposition~\ref{APF-char}.  We begin with Condition~(\ref{converge}).  Because $\Phi_{n+1}(x) = \Phi_n(\phi_{E_{n+1}/E_n}(x))$, we know that $\Phi_{n+1}(x) = \Phi_n(x)$ for all $x \leq v$, where $v$ is the $x$-coordinate of the first vertex of $\phi_{E_{n+1}/E_n}(x)$.  Let $q := \max(q_n)$, let $\epsilon$ be as in Lemma~\ref{b valuations}, and set $x_0 := \frac{\epsilon}{q}$.  We claim that $v \geq e_{E_{n+1}/K} x_0$, which will complete the verification of Condition~(\ref{converge}).   By Proposition~\ref{copoly and phi}, it suffices to show that the first vertex of $\Psi_{g_n}(x)$ has $x$-coordinate at least $x_0$, where as usual $g_n(x) := f_n(x + \pi_{n+1})$ and $f_n(x)$ is the minimal polynomial of $\pi_{n+1}$ over $E_n$.  From Lemma~\ref{b valuations}, the only contribution to the Newton copolygon of $g_n(x)$ with $y$-intercept 0 occurs with slope 
$q_n$.  All other contributions to the Newton copolygon have positive slope and $y$-intercept at least $\epsilon$.  The line $y = q_n x$ crosses the line $y = \epsilon$ at $x = \epsilon/q_n \geq \epsilon/q$, as required.

We now verify that Condition~(\ref{vertices increasing condition}) of Proposition~\ref{APF-char} holds.  
We have seen that $\Phi(x) = \Phi_n(x) $ for all $x \leq e_{E_{n+1}/K} x_0$.  If $\max(q_n) = p^s$, then $\Phi_n(x)$ has at most $ns$ vertices and so $i_{ns+1} \geq e_{E_{n+1}/K} x_0$, and in particular, the sequence $\{i_n\}$ is unbounded.  It remains to check that the $\{b_n\}$ sequence is unbounded.  Because $\Phi(x)$ is monotone increasing, it suffices to show that $\lim_{x \rightarrow \infty} \Phi(x) = \infty$.  This will follow from the claim that for any $x \geq e_{E_{n+1}/K} x_0$, we have $\Phi(x) \geq q_1 x_0 + (q_2 - 1)x_0 + \cdots + (q_{n} - 1)x_0$.  To see this, notice that between $x = e_{E_i/K}x_0$ and $x = e_{E_{i+1}/K}x_0$, the slope of $\Phi(x)$ is at least $\frac{1}{e_{E_i/K}} = \frac{1}{q_1 \cdots q_{i-1}}$.  We then compute that for $x \geq e_{E_{n+1}/K} x_0$, we have
\begin{align*}
\Phi(x) \geq 1 \cdot q_1 x_0 + \frac{1}{q_1} (q_1 q_2 - q_1) x_0 + \cdots + \frac{1}{q_1 \cdots q_{n-1}} (q_1 \cdots q_{n} - q_1\cdots q_{n-1})x_0,
\end{align*}
which completes the proof.  
\end{proof}

\begin{proposition} \label{sufficient direction}
The extension $L/K$ is strictly APF.
\end{proposition}

\begin{proof}
By Proposition~\ref{it is APF}, we know that $L/K$ is APF;  let $\Phi_n(x)$
and $\Phi(x)$
be the functions of Proposition~\ref{APF-char} and let
$V_n$ be the set of $x$-coordinates of vertices of $\Phi_n$. 
For $x_0 = \epsilon/q$ as in the proof of Proposition~\ref{it is APF}, we will prove that 
\begin{equation} \label{x-coord times slope bound}
\min_{v \in V_n} vm_v \geq x_0;
\end{equation}
it will then follow from Lemma~\ref{comparison V and V_n} that $L/K$ is strictly APF.

We will prove (\ref{x-coord times slope bound}) using induction on $n$.  
In the proof of Proposition~\ref{it is APF}, we showed that any $v \in V_2$ satisfies $v \geq q_1 x_0$;  on the other hand, the slopes of $\Phi_2(x)$ are all at least $1/q_1$.  This settles the base case $n=2$.
For the inductive step, let $v \in V_{n+1}$ and consider the following two cases:
\begin{enumerate}
\item Assume $v < e_{E_{n+1}/K} x_0$.  In this range, $\Phi_{n+1}(x) = \Phi_n(x)$ and we are finished by the inductive hypothesis.  
\item Assume $v \geq e_{E_{n+1}/K} x_0$.  Then $vm_v \geq e_{E_{n+1}/K} x_0 m_v \geq e_{E_{n+1}/K} x_0\cdot e_{E_{n+1}/K}^{-1} = x_0.$
\end{enumerate}
\end{proof}

Proposition~\ref{sufficient direction} concludes the proof that $L/K$ is strictly APF, giving
one direction of Theorem~\ref{main theorem}.  The remainder of this section is devoted to proving the converse.

We now fix an infinite and totally wildly ramified strictly APF extension $L/K$,
and let $\{K_n\}_{n\ge 1}$ be the associated tower of elementary extensions as in
Definition~\ref{elem subfield def}, so that $K_1 = K$ and $K_{n+1}/K_{n}$ is elementary of level $i_n$;
we set $q_n:=[K_{n+1}:K_n]$, so that $[K_{n+1}:K]=q_1q_2\cdots q_n$.
Let $\pi_n\in K_{n}$ be any choice of a norm-compatible family of uniformizers.\footnote{Such a choice
exists as $L/K$ is (strictly) APF.  Indeed, the norm field of $L/K$ is by definition 
$X_K(L):=\varprojlim_{E\in \E_{L/K}} E$, where $\E_{L/K}$ is the collection of finite extensions
of $K$ in $L$ and the limit is taken with respect to the Norm mappings.
For any nonzero $(\alpha_E)_E\in X_K(L)$, one defines $v(\alpha):=v_K(\alpha_K)$.
By \cite[2.2.4, 2.3.1]{Wintenberger}, one knows that $(X_K(L),v)$ is a complete, discretely
valued field with residue field $k$, and any choice of uniformizer
in $X_K(L)$ corresponds to a norm compatible sequence $(\pi_E)_{E}$ with
$\pi_E$ a uniformizer of $E$. 
}

\begin{proposition} \label{nec coef}
Let 
\[
f_n(x) = x^{q_n} + a_{n,q_n-1}x^{q_n-1} + \cdots + a_{n,1} x + (-1)^p \pi_{n}
\] 
denote the minimal polynomial of $\pi_{n+1}$ over $K_{n}$. Then the valuations of the coefficients $v_K(a_{n,i})$ for $0 < i < q_n$ are bounded below by a positive constant {\rm(}independent of $n$ and $i${\rm)}.  
\end{proposition}

\begin{proof}
We prove this by contradiction.  As $L/K$ is strictly APF, there exists $c > 0$ such that
\begin{equation} \label{str apf eqn necessary}
\inf_n \frac{i_n}{q_1 \cdots q_n} \geq c.
\end{equation}
Suppose that 
\begin{equation}
v_K(a_{n,i}) < c \label{bdd vals eqn necessary}
\end{equation}
for some $n$ and $i$.
From (\ref{str apf eqn necessary}) and (\ref{bdd vals eqn necessary}) we will reach a contradiction.

Because $K_{n+1}/K_{n}$ is elementary, from the discussion following Definition~\ref{elem subfield def} we know that the transition function $\phi_{K_{n+1}/K_{n}}(x)$ has a unique vertex $(i_n, i_n)$.  By Proposition~\ref{copoly and phi}, this means that for $g_n(x):=f_n(x+\pi_{n+1})$, the copolygon boundary function
$\Psi_{g_n}(x)$ has a unique vertex with $x$-coordinate 
$i_n/(q_1 \cdots q_n)$.  By the correspondence between Newton polygons and copolygons (see for example \cite[\S1]{Lubin}), we know that the Newton polygon of $g_n$ has exactly one segment of slope
\begin{equation} \label{Newton slope necessary}
\frac{-i_n}{q_1 \cdots q_n} \leq -c,
\end{equation}
where the inequality follows from (\ref{str apf eqn necessary}).  On the other hand, 
writing $g_n(x) = \sum_{j \geq 1} b_{n,i} x^i$ we have 
\begin{equation} \label{coefficient notation 2}
   v_K(b_{n,i}) = v_K\left(\sum_{j\ge i} a_{n,j} \binom{j}{i}\pi_{n+1}^{j-i}\right)
   = \min_{j \geq i} v_K\left( a_{n,j} \binom{j}{i}\pi_{n+1}^{j-i}\right)
\end{equation}
as the valuations of the nonzero terms in the sum are all distinct: in fact, they are all distinct modulo 
$1/(q_1 \cdots q_{n-1})$.  Now, using (\ref{bdd vals eqn necessary}), we have $v_K(b_{n,i}) \leq v_K \left(a_{n,i} \binom{i}{i} \pi_{n+1}^0 \right)< c$.    

We now compute the Newton polygon associated to $g_n$. It must pass through the point $(q_n,0)$ and by the discussion in the previous paragraph, it must pass below the point $(i,c)$.  Such a Newton polygon has slope strictly greater than (i.e., negative and smaller in absolute value than) $\frac{-c}{q_n - i} \geq -c$.   This contradicts (\ref{Newton slope necessary}).  
\end{proof}

\begin{proposition} \label{necessary bounded degrees}
With notation as in Proposition~$\ref{nec coef}$, the degrees $q_n$ are bounded above.
\end{proposition}

\begin{proof}
The proof is similar to the proof of Proposition~\ref{nec coef}.  As $L/K$ is strictly APF, we can find a positive constant $c$ such that for all $n$,
\[
\frac{i_n}{q_1 \cdots q_n} \geq c.
\]
Since $K_{n+1}/K_n$ is elementary, the Newton polygon of $f_n(x+\pi_{n+1})$ consists of a single segment with slope 
having absolute value greater than or equal to $c$.  In the notation of (\ref{coefficient notation 2}), this implies that 
\begin{equation}
c \leq \frac{v_K(b_{n,1})}{q_n-1} = \frac{v_K\left( \sum_{j\ge 1} a_{n,j} \binom{j}{1}\pi_{n+1}^{j-1} \right)}{q_n-1} \leq \frac{v_K\left( q_n\pi_{n+1}^{q_n-1} \right)}{q_n-1} 
= \frac{v_K(p) \cdot \log_p(q_n) + \frac{q_n-1}{q_1 \cdots q_n}}{q_n-1}. \label{bound on first coef}
\end{equation}
This implies $\{q_n\}_{n\ge 1}$ is bounded.
\end{proof}

\begin{remark}
Notice that in the equicharacteristic case, the term $v_K(p)$ appearing in (\ref{bound on first coef}) is $v_K(0)$, and so our argument fails.  See also Example~\ref{char p example}.  
\end{remark}

\begin{proof}[Proof of Theorem~$\ref{main theorem}$]
The content of Theorem~\ref{main theorem} is that, in order for $L/K$ to be strictly APF, it is necessary and sufficient that there exist a tower of subfields satisfying Conditions~(\ref{bdd}) and (\ref{coef}).  That an infinite totally wildly ramified extension containing such a tower of subextensions is strictly APF follows from Proposition~\ref{sufficient direction}.  That the tower of elementary subextensions of a strictly APF extension, together with any norm compatible family of uniformizers, satisfies Conditions~(\ref{bdd}) and (\ref{coef}) follows from Proposition~\ref{nec coef} and Proposition~\ref{necessary bounded degrees}.  
\end{proof}

\section{Examples} \label{examples section}

We conclude with examples which illustrate the subtlety of the APF and strictly APF
conditions.

\begin{example} \label{increasing degree example}
	Fix a sequence of positive integers $\{r_n\}_{n\ge 1}$
	and set $q_n:=p^{r_n}$.  Let $K$ be a finite extension of $\Q_p$,
	choose a uniformizer $\pi_1$ of $K$, and for $n\ge 1$ recursively choose 
	a root $\pi_{n+1}$ of $f_n(x) := x^{q_n}+\pi_1 x + (-1)^p \pi_{n}=0$.  
	Set $E_1:=K$ and for $n\ge 2$ let $E_{n+1}:=E_{n}(\pi_{n+1})$ and put $L=\cup_{n\ge 1}E_n$.

We first claim that $E_{n+1}/E_n$ is elementary of level $i_n = q_1 q_2 \cdots q_{n}/(q_n - 1)$.  As in the proof of Proposition~\ref{nec coef}, we would like to show that the Herbrand transition function 
$\phi_{E_{n+1}/E_n}(x)$ has exactly two segments: a segment of slope $1$ from $x = 0$ to $x = i_n$, and a segment of slope $1/q_n$ for $x > i_n$.  Equivalently, it suffices to show that the Newton polygon of $f_n(x + \pi_{n+1})$ has exactly one segment of slope $-i_n/e_{E_{n+1}/K}$.   (As always, we use the $v_K$ valuation for drawing Newton polygons.)

Using that $q_n$ is a power of $p$, the binomial theorem shows that the Newton polygon of 
$f_n(x + \pi_{n+1})$ is the lower convex hull of the collection of vertices containing $(1,1)$, $(q_n, 0)$, and other vertices with $y$-coordinate at least $1$.  Hence the Newton polygon consists of a single segment of slope $-1/(q_n - 1)$.  Thus $i_n = q_1 q_2 \cdots q_{n}/(q_n - 1)$, as desired.  

Notice that the $\{i_n\}_n$ is strictly increasing.  We may thus use 
Corollary~\ref{WintCor} to analyze the extension $L/K$.  Define $b_n$ as in (\ref{bformula}).  
Substituting $i_n = q_1 q_2 \cdots q_{n}/(q_n - 1)$ into the definition of the terms $b_n$, we find
\[
b_n = \frac{q_1}{q_1-1} +\sum_{k=2}^n \left( \frac{q_k}{q_k-1} - \frac{1}{q_{k-1}-1} \right),
\]
and it follows from Corollary~\ref{WintCor} that $L/K$ is APF for every choice of $q_n$ (i.e., for every choice of $r_n$).  On the other hand, by Definition~\ref{strAPF}, $L/K$ is strictly APF if and only if 
\[
\inf_{n>0}  \frac{i_n}{[E_{n+1}:K]}
		 = \inf_{n>0} \frac{1}{(q_n-1)} > 0.
\]
In other words, the extension $L/K$ is strictly APF if and only if the degrees $q_n$ are bounded above.  
\end{example}

\begin{example}
	Fix an increasing sequence $\{s_n\}_{n\ge 1}$ of positive integers
	and let $K$ be a finite extension of $\Q_p$ with absolute ramification 
	index $e$.  Choose a uniformizer $\pi_1$ of $K$, set $E_1:=K$ and for $n\ge 2$
	recursively choose $\pi_{n+1}$ a root of $x^p + \pi_{n}^{s_n}x-\pi_{n}=0$
	and put $E_{n+1}:=E_{n}(\pi_{n+1})$.
	Set $L=\cup_{n\ge 1}E_n$.
	
	As in Example~\ref{increasing degree example}, if we assume that $s_n \leq p^{n-1}e$,
	we compute that $E_{n+1}/E_n$ is elementary of level $i_n= ps_n/(p-1)$, and because we have chosen $s_n$ to be an increasing sequence, we may again apply Corollary \ref{WintCor}.
With $b_n$ as in (\ref{bformula}), we compute 
	\begin{equation*}
		b_n =\frac{ps_1}{p-1} + \frac{p}{p-1}\sum_{k=2}^n \frac{s_k-s_{k-1}}{p^{k-1}}.
	\end{equation*}
As the following examples illustrate, whether or not the extension $L/K$ is APF, strictly APF, or neither, depends crucially on the choice of $s_n$:
\begin{enumerate}
\item If one takes $s_n = n$, then the $b_n$ terms are increasing but bounded.  In this case, the extension $L/K$ is not APF.
\item Assume $p \geq 5$ and take $s_n=\lfloor p^{n-1}/n\rfloor$.  Then 
$\{i_n\}_{n\ge 1}$ is strictly increasing (using the hypothesis $p\ge 5$).
Moreover, the sequence $\{b_n\}_{n\ge 1}$ is increasing and unbounded and so $L/K$ is APF, but 
\[
\inf_{n > 0}  \frac{i_n}{[E_{n+1}:K]} = \inf_{n>0}\frac{s_n}{p^{n-1}(p-1)} = 0,
\]
and so $L/K$ is APF but not strictly APF.
\item If we take $s_n = p^{n-1}$, then $\{b_n\}_{n\ge 1}$ is increasing and unbounded, and 
\[
\inf_{n > 0} \frac{s_n}{p^{n-1}(p-1)} = \frac{1}{p-1} > 0,
\]
so $L/K$ is strictly APF.
\end{enumerate}
\end{example}

\begin{remark} \begin{enumerate}
\item Assume $L/K$ is an infinite totally wildly ramified strictly APF extension.  One cannot expect that Condition~(\ref{bdd}) of Theorem~\ref{main theorem} hold for \emph{every} tower of subextensions $\{E_n\}$.  For example, 
for $K$ a finite extension of $\Q_p$ and $\pi_1$ a uniformizer of $K$, 
consider the extension $L/K$ formed by recursively extracting roots of the polynomials $f_n(x) = x^{p^n} - \pi_{n}$.  These polynomials determine the same extension as the polynomials $f_n(x) = x^p - \pi_{n}$; however the former collection of polynomials has unbounded degrees, while the degrees in the latter collection are all equal to $p$.  
\item The authors do not know whether Condition~(\ref{coef}) of Theorem~\ref{main theorem} holds for \emph{every} tower of subextensions and every norm-compatible choice of uniformizers.  
\end{enumerate}
\end{remark}

\begin{example} \label{char p example}
Here we give an example to show that the full strength of our theorem does not hold in characteristic~$p$; see Remark~\ref{char p remark} for positive results.  
Assume $K$ is a local field of characteristic~$p$, and let $\pi_1 \in K$ denote a uniformizer.  Consider the polynomials
\[
f_n(x) = x^{p^n} + \pi_1^{p^n} x -\pi_{n},  
\]
and let $\pi_{n+1}$ denote a root of $f_n(x)$.  Set $E_{n+1} := E_n(\pi_{n+1})$ and $L := \cup E_n$.   We claim that $L/K$ is strictly APF, and that $\{E_n\}$ is the associated tower of elementary extensions.  Because the degrees $\deg f_n = p^n$ are unbounded, this shows that Theorem~\ref{main theorem} is not true for local fields of characteristic~$p$.  

We compute $f_n(x + \pi_{n+1}) = x^{p^n} + \pi_1^{p^n}x$ and so the Newton polygon is a single segment with slope
\[
\frac{-p^n}{p^n - 1} = \frac{-i_n}{p \cdot p^2  \cdots  p^{n}},  
\]
which implies
\[
i_n = \frac{p \cdot p^2  \cdots  p^n \cdot p^n}{p^n - 1}. 
\]
This is a strictly increasing sequence, so we can apply Corollary \ref{WintCor} as above.  One checks that the sequence $\{b_n\}$ defined by (\ref{bformula}) is increasing and unbounded 
and
\[
\inf \frac{i_n}{[E_{n+1} : E_1]} > 0.
\]
Corollary~\ref{WintCor} then shows that $L/K$ is strictly APF, as desired.  
\end{example}

\begin{remark} Theorem~\ref{main theorem} is perhaps better suited to producing strictly APF extensions 
than to establishing whether a given extension $L/K$ is strictly APF.  For example, consider the extension $\Q_p(\mu_{p^{\infty}}, p^{1/p^{\infty}})/\Q_p$.  This is a Galois extension with Galois group a $p$-adic Lie group, hence is strictly APF extension by Sen's theorem \cite[\S4]{Sen72}.   However, the authors do not know how to verify this 
fact using Theorem~\ref{main theorem}, because we do not know how to select a tower
$\{E_n\}_{n\ge 1}$ and a norm compatible family of uniformizers $\{\pi_n\}_{n\ge 1}$
which is amenable to explicitly computing
the polynomials $f_n$ as in the statement of Theorem \ref{main theorem}. 
%identify a norm compatible family of uniformizers in the tower $\left\{\Q_p(\mu_{p^i}, \pi^{1/p^j})\right\}$.
\end{remark}

\subsection*{Acknowledgments} 
The first author is supported by an NSA ``Young Investigator" grant (H98230-12-1-0238).
The second author is partially supported by the Danish National Research Foundation through the Centre for Symmetry and Deformation (DNRF92).  The first two authors largely worked on this project together while visiting Lars Hesselholt at Nagoya University and Takeshi Saito at the University of Tokyo; special thanks to them for their support and useful discussions.

\bibliography{canonical}
\bibliographystyle{plain}

\end{document}